\documentclass{amsart}

\usepackage{amssymb}
\usepackage[hidelinks]{hyperref}
\usepackage[textsize=footnotesize]{todonotes}
\usepackage{tikz}
\usetikzlibrary{cd}
\usepackage{float}
\usepackage{xcolor}
\usepackage{todonotes}
\usepackage{framed}

\theoremstyle{plain}
\newtheorem{theorem}{Theorem}[section]
\newtheorem{lemma}[theorem]{Lemma}
\newtheorem{corollary}[theorem]{Corollary}
\newtheorem{proposition}[theorem]{Proposition}

\theoremstyle{definition}
\newtheorem{definition}[theorem]{Definition}
\newtheorem{example}[theorem]{Example}

\newtheorem{remark}[theorem]{Remark}

\theoremstyle{remark}

\newcommand{\bR}{\mathbb{R}}
\newcommand{\R}{\bR}

\newcommand{\cA}{\mathcal{A}}
\newcommand{\cB}{\mathcal{B}}

\newcommand{\cI}{\mathcal{I}}
\newcommand{\I}{\cI}
\newcommand{\cJ}{\mathcal{J}}
\newcommand{\J}{\cJ}

\newcommand{\cM}{\mathcal{M}}
\newcommand{\cN}{\mathcal{N}}
\newcommand{\cP}{\mathcal{P}}

\newcommand{\bnumber}{\mathfrak{b}}
\newcommand{\bb}{\bnumber}

\DeclareMathOperator{\non}{non}

\newcommand{\fin}{\mathrm{Fin}}
\newcommand{\Fin}{\mathrm{Fin}}

\newcommand{\ED}{\mathcal{ED}}

\newcommand{\Exh}{\mathrm{Exh}}

\begin{document}


\title[Egorov ideals]{Egorov ideals}


\author[Adam Kwela]{Adam Kwela}
\address[Adam Kwela]{Institute of Mathematics\\ Faculty of Mathematics\\ Physics and Informatics\\ University of Gda\'{n}sk\\ ul.~Wita  Stwosza 57\\ 80-308 Gda\'{n}sk\\ Poland}
\email{Adam.Kwela@ug.edu.pl}
\urladdr{http://mat.ug.edu.pl/~akwela}


\date{\today}


\subjclass[2010]{Primary: 
}




\begin{abstract}
We study Egorov ideals, that is ideals on $\omega$ for which the Egorov’s theorem for ideal versions of pointwise and uniform convergences holds. We show that a non-pathological $\bf{\Sigma^0_2}$ ideal is Egorov if and only if it is countably generated. In particular, up to isomorphism, there are only three non-pathological $\bf{\Sigma^0_2}$ Egorov ideals. On the other hand, we construct $2^\omega$ pairwise non-isomorphic Borel Egorov ideals. Moreover, we characterize when a product of ideals is Egorov.
\end{abstract}


\maketitle


\section{Introduction}

Let $\lambda$ denote the Lebesgue measure on $\R$. In this paper by "measurable" we always mean "Lebesgue measurable". The classical Egorov's theorem (see e.g.~\cite[Proposition 3.1.4]{MR3098996}) says that if $(f_n)_{n\in\omega}$ is a sequence of measurable real-valued functions defined on $[0,1]$ which is pointwise convergent to a function $f:[0,1]\to\R$, then for every $\eta>0$ there is a measurable set $M\subseteq[0,1]$ such that $\lambda([0,1]\setminus M)<\eta$ and $(f_n\restriction M)_{n\in\omega}$ is uniformly convergent to $f\restriction M$.

The main point of this paper is to examine ideals on $\omega$ for which the ideal version of Egorov's theorem holds. To state it precisely, we need to introduce ideal convergences.

Let $\I$ be an ideal on $\omega$. A sequence $( x_n)_{n\in\omega}$ in a topological space $X$ is \emph{$\I$-convergent to} $x\in X$ ($x_n\xrightarrow{\I}x$) if for each open neighbourhood $U$ of $x$ we have: $$\{n\in\omega:\ x_n\notin U\}\in\I.$$ 
A sequence of real-valued functions $(f_n)_{n\in\omega}$ defined on a topological space $X$ is:
\begin{itemize}
\item \emph{$\I$-pointwise convergent to} $f\in \mathbb{R}^X$ ($f_n\xrightarrow{\text{$\I$-p}}f$) if $f_n(x)\xrightarrow{\I}f(x)$ for all $x\in X$;
\item \emph{$\I$-uniformly convergent to} $f\in \mathbb{R}^X$ ($f_n\xrightarrow{\text{$\I$-u}}f$) if for each $\varepsilon>0$ we have:
$$\left\{n\in\omega:\ |f_n(x)-f(x)|\geq\varepsilon\text{ for some }x\in X\right\}\in\I.$$ 
\end{itemize}

The following definition will be the main object of our studies.

\begin{definition}
\label{def}
Let $\I$ be an ideal on $\omega$. We say that  $\I$ is Egorov if for any sequence $(f_n)_{n\in\omega}$ of measurable real-valued functions defined on $[0,1]$ which is $\I$-pointwise convergent to a measurable function $f:[0,1]\to\R$ and for every $\eta>0$ there is a measurable set $M\subseteq[0,1]$ such that $\lambda([0,1]\setminus M)<\eta$ and $(f_n\restriction M)_{n\in\omega}$ is $\I$-uniformly convergent to $f\restriction M$.
\end{definition}

Note that many properties of ideals closely related to the above one have been considered in the literature so far. Egorov ideals in our sense were considered by Mro\.{z}ek in \cite{Nikodem-doktorat}, Korch in \cite{Korch} and Repick\'{y} in \cite{Repicky}. Mro\.{z}ek in \cite{Nikodem-doktorat} considered the condition from Definition \ref{def} without the assumption that the limit function $f$ is measurable (a pointwise limit of measurable functions is measurable, but this does not have to be true for ideal pointwise limits -- see \cite[Fact 1]{RL}). Kadets and Leonov in \cite{KadetsLeonov} as well as Mro\.{z}ek in \cite{Nikodem} studied ideals satisfying the condition from Definition \ref{def} for every finite measure space (not only for $[0,1]$). Finally, Korch in \cite{Korch} and Repick\'{y} in \cite{Repicky} investigated ideals satisfying the generalized Egorov's theorem, i.e., condition from Definition \ref{def} without measurability of functions $f_n$ and $f$ (such property, even in the case of classic convergence, is independent from ZFC). In particular, every ideal which is Egorov in the sense of Kadets and Leonov or generalized Egorov in the sense of Korch and Repick\'{y}, is also Egorov in our sense.

We decided to restrict ourselves to the measure space $[0,1]$ and to measurable functions for two reasons: it makes the paper easier to read and the property defined in Definition \ref{def} already seems to be very rare. However, most of our results are true also when one considers the above modifications of Definition \ref{def} (in the paper we indicate results with proofs valid only for Lebesgue measure and measurable functions). In particular, this applies to our two main results: characterization of Egorov ideals in the class of non-pathological $\bf{\Sigma^0_2}$ ideals (see Corollary \ref{GenEgorov-cor1}) and existence of $2^\omega$ pairwise non-isomorphic Borel Egorov ideals (see Corollary \ref{GenEgorov-cor2}).

The paper is organized as follows. In Section \ref{sec:preliminaries} we recall some set-theoretic notation and rather standard facts about ideals on $\omega$. In Section \ref{sec:characterizations}, we collect some known facts about Egorov ideals and apply them to give examples of well-known ideals that are or are not Egorov. Moreover, we prove a characterization of Egorov ideals that will be needed in one of the proofs from Section \ref{sec:products}. Section \ref{sec:Fsigma} is devoted to the first of our main results: we show that a non-pathological $\bf{\Sigma^0_2}$ ideal is Egorov if and only if it is countably generated, and conclude that, up to isomorphism, there are only three non-pathological $\bf{\Sigma^0_2}$ Egorov ideals (see Theorem \ref{thmFsigma}). In Section \ref{sec:RK}, using Rudin-Keisler order on ideals, we deal with pathological $\bf{\Sigma^0_2}$ ideals -- we show that being an Egorov ideal is downward closed in the Rudin-Keisler order, prove some facts about Rudin-Keisler order on $\bf{\Sigma^0_2}$ ideals and justify that the Mazur's ideal, i.e., most well-known example of a pathological $\bf{\Sigma^0_2}$ ideal, is not Egorov. In Section \ref{sec:products} we characterize when a product of ideals is Egorov. Those results are closely related to the results of \cite{Repicky}. Finally, in Section \ref{sec:uncountable} we prove the second of our main results by showing that there are $2^\omega$ pairwise non-isomorphic Borel Egorov ideals (see Theorem \ref{thm-uncountably}).


\section{Preliminaries}
\label{sec:preliminaries}

By $\omega$ we denote the set of all natural numbers, i.e., $\omega=\{0,1,\ldots\}$.
We identify a natural number $n$ with the set $\{0, 1,\dots , n-1\}$. If $A$ and $B$ are two sets then by $A^B$ we denote the family of all functions $f:B\to A$. The symbol $\omega^{<\omega}$ ($2^{<\omega}$) stands for all finite sequences of natural numbers (finite $0-1$ sequences, respectively).

If $s=(s(0),s(1),\ldots,s(n)),t=(t(0),t(1),\ldots,t(k))\in \omega^{<\omega}$ and $x\in\omega^\omega$, then the concatenation of $s$ with $t$ ($s$ with $x$) is
$s^\frown t=(s(0),s(1),\ldots,s(n),t(0),t(1),\ldots,t(k))$
($s^\frown x=(s(0),s(1),\ldots,s(n),x(0),x(1),\ldots)$, respectively) and we write $s\subseteq t$ ($s\subseteq x$) whenever $n\leq k$ and $s(i)=t(i)$, for all $i\leq n$ ($s(i)=x(i)$, for all $i\leq n$, respectively). By $|s|$ we denote the length of a sequence $s\in\omega^{<\omega}$.

We say that a family $\cA\subseteq\cP(\omega)$ is almost disjoint if $A\cap B$ is finite for all $A,B\in\cA$, $A\neq B$.

For $A\subseteq[0,1]$ by $\chi_A$ we denote the characteristic function of the set $A$, i.e., $\chi_A:[0,1]\to\mathbb{R}$ is given by:
$$\chi_A(x)=\begin{cases}
    1, & \text{if }x\in A,\\
    0, & \text{if }x\in [0,1]\setminus A,
\end{cases}$$
for all $x\in[0,1]$.


\subsection{Ideals}

An \emph{ideal on a set $X$} is a family $\I\subseteq\cP(X)$ that satisfies the following properties:
\begin{enumerate}
\item if $A,B\in \I$ then $A\cup B\in\I$,
\item if $A\subseteq B$ and $B\in\I$ then $A\in\I$,
\item $\I$ contains all finite subsets of $X$,
\item $X\notin\I$.
\end{enumerate}
In this paper we are only interested in ideals on countable sets, i.e., $X$ will always be countable.

An ideal $\I$ on $X$ is \emph{tall} if for every infinite $A\subseteq X$ there is an infinite $B\in\I$ such that $B\subseteq A$. An ideal $\I$ on $X$ is a \emph{P-ideal} if for any countable family $\cA\subseteq\I$ there is $B\in \I$ such that $A\setminus B$ is finite for every $A\in \cA$. If $\I$ is an ideal on $X$, then $\I^+=\{X\setminus A:\ A\in\I\}$ is its \emph{coideal} and for each $B\in\I^+$ the family $\I\restriction B=\{A\cap B:\ A\in\I\}$ is an ideal on $B$. We say that \emph{$\I$ is generated by a family $\cA\subseteq\cP(X)$}, if $\cA\subseteq\I$ and for each $B\in\I$ there is a finite family $\cA'\subseteq\cA$ such that $B\subseteq\bigcup\cA'$. An ideal is \emph{countably generated}, if there is a countable family generating it.

We treat the power set $\mathcal{P}(X)$ as the space $2^X$ of all functions $f:X\rightarrow 2$ (equipped with the product topology, where each space $2=\left\{0,1\right\}$ carries the discrete topology) by identifying subsets of $X$ with their characteristic functions. Thus, we can talk about descriptive complexity of subsets of $\mathcal{P}(X)$ (in particular, of ideals on $X$). 

It is known that there are no $\bf{\Pi^0_2}$ ideals and every analytic P-ideal is $\bf{\Pi^0_3}$ (see \cite[Lemma 1.2.2 and Theorem 1.2.5]{Farah}). 

Let $\I$ and $\J$ be ideals on $X$ and $Y$, respectively. We say that:
\begin{itemize}
    \item $\I$ and $\J$ are \emph{isomorphic}, if there is a bijection $f:Y\to X$ such that:
    $$A\in\I\ \Leftrightarrow\ f^{-1}[A]\in\J,$$
    for every $A\subseteq X$;
    \item \emph{$\I$ is below $\J$ in the Rudin-Keisler order} ($\I\leq_{RK}\J$), if there is a function $f:Y\to X$ (not necessarily bijection) such that:
    $$A\in\I\ \Leftrightarrow\ f^{-1}[A]\in\J,$$
    for every $A\subseteq X$;
    \item \emph{$\I$ is below $\J$ in the Rudin-Blass order} ($\I\leq_{RB}\J$) if $\I\leq_{RK}\J$ with the witnessing function being finite-to-one.
\end{itemize}

For simplicity, we will mostly work with ideals on $\omega$. However, it is easy to see that many properties of ideals (such as tallness, being an Egorov ideal, being a P-ideal, descriptive complexity of an ideal, etc.) are preserved under isomorphisms of ideals. Thus, our results are valid for ideals on any countable set.

\subsection{Sums and products of ideals}

If $\{X_t:\ t\in T\}$ is a family of sets, then $\sum_{t\in T}X_t=\{(t,x):\ t\in T,x\in X_t\}$ is its disjoint sum. The vertical section of a set $A\subseteq \sum_{t\in T}X_t$ at a point $t\in T$ is defined 
by $A_{(t)} = \{x\in X_t:\ (t,x)\in A\}$. 

For a family of ideals $\{\I_t:\ t\in T\}$, where $T$ is a countable set, and ideal $\I$ and $\J$ on $S$ and $T$, respectively, we define the following new ideals:
\begin{enumerate}
	\item $\sum_{t\in T}\I_t=\{M\subseteq \sum_{t\in T}X_t:\ \forall_{t\in T}\ M_{(t)}\in\I_t\}$;
 	\item $\sum^J_{t\in T}\I_t=\{M\subseteq \sum_{t\in T}X_t:\ \{t\in T:\ M_{(t)}\notin\I_t\}\in\J\}$;
    \item $\I\otimes\J=\sum^\I_{s\in S}\J=\{M\subseteq S\times T:\ \{s\in S:\ M_{(s)}\notin\J\}\in\I\}$;
	\item
$\I\otimes \{\emptyset\} = \{M\subseteq S\times \omega:\ \{s\in S:\ M_{(s)}\ne\emptyset\}\in \I\}$;
	\item
$\{\emptyset\} \otimes  \J= \{M\subseteq \omega\times T:\ M_{(n)} \in \J\text{ for all }n\in \omega\}$;
    \item $\I\oplus\J=\{M\subseteq(\{0\}\times S)\cup(\{1\}\times T):\ M_{(0)}\in\I\text{ and }M_{(1)}\in\J\}$;
    \item $\I\oplus\cP(\omega)=\{M\subseteq(\{0\}\times S)\cup(\{1\}\times \omega):\ M_{(0)}\in\I\}$.
\end{enumerate}

\subsection{Submeasures}

A function $\phi:\mathcal{P}(\omega)\to[0,\infty]$ is called a \emph{submeasure} if $\phi(\emptyset)=0$, $\phi(\{n\})<\infty$, for each $n\in\omega$, and:
$$\forall_{A,B\subseteq\omega}\ \phi(A)\leq\phi(A\cup B)\leq\phi(A)+\phi(B).$$
A submeasure $\phi$ is \emph{lower semicontinuous} (lsc, in short) if:
$$\forall_{A\subseteq\omega}\ \phi(A)=\lim_{n\to\infty}\phi(A\cap n).$$

Mazur in \cite[Lemma 1.2]{Mazur} proved that an ideal on $\omega$ is $\bf{\Sigma^0_2}$ if and only if it is of the form:
$$\Fin(\phi)=\left\{A\subseteq\omega:\ \phi(A)<\infty\right\}$$
for some lower semicontinuous submeasure $\phi$ such that $\omega\notin\Fin(\phi)$ (see also \cite[Theorem 1.2.5]{Farah}).

Solecki in \cite[Theorem 3.1]{SoleckiExh} showed that an ideal on $\omega$ is an analytic P-ideal if and only if it is of the form:
$$\Exh(\phi)=\left\{A\subseteq\omega:\ \lim_{n\to\infty}\phi(A\setminus n)=0\right\}$$
for some lower semicontinuous submeasure $\phi$ such that $\omega\notin\Exh(\phi)$ (see also \cite[Theorem 1.2.5]{Farah}). 

A submeasure $\phi$ is \emph{non-pathological}, if
$$\phi(A)=\sup\{\mu(A):\ \mu\text{ is a measure such that }\mu(B)\leq\phi(B)\text{ for all }B\subseteq\omega\},$$
for all $A\subseteq\omega$. We say that an $\bf{\Sigma^0_2}$ ideal (analytic P-ideal) is \emph{non-pathological}, if it is of the form $\Fin(\phi)$ ($\Exh(\phi)$, respectively), for some non-pathological submeasure $\phi$.

\subsection{Examples of ideals}
	
\begin{example}\
\begin{itemize} 
\item $\fin$ is the ideal of all finite subsets of $\omega$. It is a non-tall $\bf{\Sigma^0_2}$ P-ideal (\cite[Example 1.2.3]{Farah}).
\item $\fin\oplus\cP(\omega)$ is a non-tall $\bf{\Sigma^0_2}$ P-ideal (\cite[Example 1.2.3]{Farah}).
\item $\fin\otimes\{\emptyset\}$ is a $\bf{\Sigma^0_2}$ ideal that is not tall and not a P-ideal (\cite[Example 1.2.3]{Farah}).
\item $\{\emptyset\}\otimes\fin$ is a non-tall $\bf{\Pi^0_3}$ P-ideal (\cite[Example 1.2.3]{Farah}).
\item $\fin^2=\fin\otimes\fin$ is a tall $\bf{\Sigma^0_4}$ non-P-ideal (\cite[Proposition 6.4]{Debs}).
\item $\fin^{n+1}=\fin\otimes\fin^n$, for $n\in\omega\setminus\{0\}$, are tall $\bf{\Sigma^0_{2n}}$ non-P-ideals (\cite[Proposition 6.4]{Debs}). These ideals were introduced in \cite{Katetov} (see also \cite{Kat2} and \cite{Kat1}).
\item If $(c_n)_{n\in\omega}$ is a sequence of positive reals such that $\sum_{n\in\omega}c_n=\infty$, then $\I_{c_n}=\{A\subseteq\omega:\ \sum_{n\in A}c_n<\infty\}$ is a $\bf{\Sigma^0_2}$ P-ideal called a \emph{summable ideal}. A summable ideal $\I_{c_n}$ is tall if and only if $\lim_{n\to\infty}c_n=0$. In particular, $\I_{1/n}=\{A\subseteq\omega:\ \sum_{n\in A}\frac{1}{n+1}<\infty\}$ is a tall summable ideal (\cite[Example 1.2.3]{Farah}). 
\item $\I_{d}=\{A\subseteq\omega: \lim\frac{|A\cap n|}{n+1}=0\}$ is the \emph{ideal of sets of asymptotic density zero}. It is a tall $\bf{\Pi^0_3}$ P-ideal (\cite[Example 1.2.3]{Farah}).
\end{itemize}
\end{example}


\section{Characterizations and examples}
\label{sec:characterizations}

\begin{remark}
\label{rem-ex}
By \cite[Corollary 2.9]{KadetsLeonov}, every countably generated ideal is Egorov. By a result of Mro\.{z}ek \cite[Corollary 3.2 and Example 3.3]{Nikodem}, among analytic P-ideals that are not pathological the only Egorov ideals are those isomorphic to $\Fin$, $\Fin\oplus\cP(\omega)$ or $\{\emptyset\}\otimes\fin$ (note that \cite{Nikodem} uses a slightly different definition of isomorphism of ideals, which actually comes from \cite{Farah}, and for that reason $\Fin\oplus\cP(\omega)$ is not explicitly included in \cite[Corollary 3.2]{Nikodem}). Using those results together with \cite[Theorem 3.25]{Nikodem-doktorat} (see also \cite[Section 6]{Korch} and \cite[Theorems 1.3 and 3.2(7)]{Repicky}), we see that $\Fin$, $\fin\oplus\cP(\omega)$, $\fin\otimes\{\emptyset\}$, $\{\emptyset\}\otimes\fin$, $\fin^{n+1}$ for all $n\in\omega$, are Egorov, while $\I_{d}$ and summable ideals are not Egorov.
\end{remark}

In Section \ref{sec:products} we will give two other examples of Egorov ideals, that have been considered in the literature in completely different contexts. 

The following two results characterizing Egorov ideals are essentially due to Kadets and Leonov.

\begin{proposition}
\label{restrictions}
$\I$ is Egorov if and only if $\I\restriction A$ is Egorov, for every $A\in\I^+$.
\end{proposition}

\begin{proof}
If $\I$ is not Egorov, then $\I\restriction\omega=\I$ is not Egorov. On the other hand, suppose that $\I\restriction A$ is not Egorov, for some $A\in\I^+$. If $\omega\setminus A\in\I$, then $\I$ is not Egorov as it is isomorphic to $\I\restriction A$ (see for instance \cite[Proposition 1.2(b)]{homogeneous}). If $\omega\setminus A\in\I^+$ then it is easy to see that $\I$ is isomorphic to $(\I\restriction A)\oplus(\I\restriction (\omega\setminus A))$, hence, $\I$ is not Egorov by \cite[Proposition 2.4]{KadetsLeonov}.
\end{proof}

\begin{theorem}[{\cite[Theorem 2.1]{KadetsLeonov}}]
\label{characteristic}
Let $\I$ be an ideal on $\omega$. Then $\I$ is Egorov if and only if for any sequence $(X_n)_{n\in\omega}$ of measurable subsets of $[0,1]$ such that $(\chi_{X_n})_{n\in\omega}$ is $\I$-pointwise convergent to $0$ and for any $\eta>0$ there is a measurable set $M\subseteq[0,1]$ such that $\lambda(M)\leq\eta$ and $\{n\in\omega:\ X_n\not\subseteq M\}\in\I$.
\end{theorem}

Next result eliminates the parameter $\eta$ from the previous one. In the proof we use regularity of the Lebesgue measure -- we do not know whether this result is true for Egorov ideals in the sense of Kadets and Leonov.

\begin{theorem}
\label{onehalf}
Let $\I$ be an ideal on $\omega$. Then $\I$ is Egorov if and only if for any sequence $(X_n)_{n\in\omega}$ of measurable subsets of $[0,1]$ such that $(\chi_{X_n})_{n\in\omega}$ is $\I$-pointwise convergent to $0$ there is a measurable set $M\subseteq[0,1]$ such that $\lambda(M)\leq\frac{1}{2}$ and $\{n\in\omega:\ X_n\not\subseteq M\}\in\I$.
\end{theorem}

\begin{proof}
The implication $(\implies)$ follows from Theorem \ref{characteristic}. We will shown the opposite one.

Assume that $\I$ is not Egorov. By Theorem \ref{characteristic}, there are $\eta>0$ and a sequence $(X_n)_{n\in\omega}$ of measurable subsets of $[0,1]$ such that $(\chi_{X_n})_{n\in\omega}$ is $\I$-pointwise convergent to $0$, but for any measurable set $M\subseteq[0,1]$, if $\lambda(M)\leq\eta$, then $\{n\in\omega:\ X_{n}\not\subseteq M\}\notin\I$. 

Suppose towards contradiction that for any sequence $(Y_n)_{n\in\omega}$ of measurable subsets of $[0,1]$ such that $(\chi_{Y_n})_{n\in\omega}$ is $\I$-pointwise convergent to $0$, there is a measurable set $M\subseteq[0,1]$ with $\lambda(M)\leq\frac{1}{2}$ and $\{n\in\omega:\ X_{n}\not\subseteq M\}\in\I$. Let $k\in\omega$ be such that $(\frac{2}{3})^{k+1}<\eta$. 

We will inductively construct $m_n\in\omega$, for $n\leq k$, $A_n\in\I$, for $i\leq m_k$, and closed sets $F_i\subseteq[0,1]$, for $i\leq m_k$, such that:
\begin{itemize}
\item[(a)] $m_0=0$ and $m_{n}\leq m_{n+1}$, for all $n<k$;
\item[(b)] $\lambda(\bigcup_{i\leq m_n}F_i)\geq 1-(\frac{2}{3})^{n+1}$, for all $n\leq k$;
\item[(c)] $A_{i}=\{j\in\omega:\ X_{j}\cap F_{i}\neq\emptyset\}\in\I$, for all $i\leq m_k$.
\end{itemize}

At first step, put $m_0=0$. By our assumption, there is a measurable $M_0\subseteq[0,1]$ with $\lambda(M_0)\leq\frac{1}{2}$ and $\{n\in\omega:\ X_{n}\not\subseteq M_0\}\in\I$. Since $[0,1]\setminus M_0$ is measurable and $\lambda([0,1]\setminus M_0)\geq\frac{1}{2}$, by regularity of Lebesgue measure, there is a closed set $F_0\subseteq [0,1]$ such that $F_0\cap M_0=\emptyset$ and $\lambda(F_0)\geq\frac{1}{3}=1-\frac{2}{3}$. Observe that $A_0=\{n\in\omega:\ X_{n}\cap F_0\neq\emptyset\}\subseteq \{n\in\omega:\ X_{n}\not\subseteq M_0\}\in\I$.

If $m_i$, for all $i<n\leq k$, as well as $A_i$ and $F_i$, for all $i\leq m_{n-1}$, are already defined, let $\{U_{i}:\ i\in\omega\setminus(m_{n-1}+1)\}$ be open connected components of $[0,1]\setminus \bigcup_{i\leq m_{n-1}}F_i$ (if the latter set has only finitely many connected components, just put $U_i=\emptyset$ for almost all $i$). In particular, each $U_i$ is an open subinterval of $[0,1]$ and $\sum_{i>m_{n-1}}\lambda(U_i)\leq(\frac{2}{3})^{n}$. Hence, there is $m_n\in\omega$, $m_n>m_{n-1}$, such that $\sum_{m_{n-1}<i\leq m_n}\lambda(U_i)\geq\frac{5}{6}\lambda(\bigcup_{i>m_{n-1}}U_{i})$.

For each $m_{n-1}<i\leq m_n$ let $f_{i}:\overline{U_i}\to[0,1]$ be a linear function and put $X_j^{i}=f_{i}[X_j\cap U_{i}]$. Then each $X_j^{i}$ is measurable and $(\chi_{X_j^{i}})_{j\in\omega}$ is $\I$-pointwise convergent to $0$. By our assumption, for each $m_{n-1}<i\leq m_n$ there is a measurable $M_{i}\subseteq[0,1]$ with $\lambda(M_{i})\leq\frac{1}{2}$ and $\{j\in\omega:\ X_j^{i}\not\subseteq M_{i}\}\in\I$. Note that $\lambda(f^{-1}[M_i])\leq\frac12\lambda(U_i)$. By regularity of Lebesgue measure, there are closed sets $F_{i}\subseteq U_i$ such that $f_i[F_i]\cap M_{i}=\emptyset$ and $\lambda(F_i)\geq \frac{2}{5}\lambda(U_i)$. Observe that $A_i=\{j\in\omega:\ X_{j}\cap F_i\neq\emptyset\}=\{j\in\omega:\ (X_{j}\cap U_i)\cap F_i\neq\emptyset\}\subseteq \{j\in\omega:\ X_j^{i}\not\subseteq M_i\}\in\I$. Moreover, 
$$\lambda\left(\bigcup_{i\leq m_n}F_i\right)=\lambda\left(\bigcup_{i\leq m_{n-1}}F_i\right)+\sum_{m_{n-1}<i\leq m_n} \lambda(F_i)\geq$$
$$1-\left(\frac{2}{3}\right)^{n}+\frac{2}{5}\sum_{m_{n-1}<i\leq m_n} \lambda(U_i)\geq 1-\left(\frac{2}{3}\right)^{n}+\frac{1}{3}\lambda\left(\bigcup_{i>m_{n-1}}U_{i}\right)=1-\left(\frac{2}{3}\right)^{n+1}.$$
Thus, the construction is finished.

Put $M=[0,1]\setminus\bigcup_{i\leq m_k}F_i$. Then $M$ is measurable and $\lambda(M)\leq\left(\frac{2}{3}\right)^{k+1}<\eta$ and $\{i\in\omega:\ X_i\not\subseteq M\}=\bigcup_{i\leq m_k}A_i\in\I$, which contradicts the choice of $(X_i)_{i\in\omega}$ and finishes the proof.
\end{proof}

\section{Non-pathological \texorpdfstring{$\bf{\Sigma^0_2}$}{F } ideals}
\label{sec:Fsigma}

In this section we will show that the only non-pathological $\bf{\Sigma^0_2}$ Egorov ideals are those isomorphic to $\Fin$, $\Fin\oplus\cP(\omega)$ or $\Fin\otimes\{\emptyset\}$. We will work with trees, so we need to introduce some notation.

\begin{definition}
Let $c\in\omega^\omega$. Denote:
\begin{itemize}
    \item by $T_{c}$ the family of all $s=(s(0),s(1),\ldots,s(n))\in \omega^{<\omega}$ such that $s(i)<c(i)$ for all $i\leq n$;
    \item by $\hat{T}_{c}$ the family of all $x\in \omega^{\omega}$ such that $x(i)<c(i)$ for all $i\leq n$;
    \item $A_x=\{x\restriction n:\ n\in\omega\}\subseteq T_c$, for each $x\in \hat{T}_c$.
\end{itemize}
\end{definition}

Now we will show that a non-pathological $\bf{\Sigma^0_2}$ ideal, which is not countably generated, cannot be Egorov. We decided to divide this proof into two lemmas, as in the first one we do not need the assumption that our ideal is non-pathological.

\begin{lemma}
\label{lem1Fsigma}
If $\I=\Fin(\phi)$ is an $\bf{\Sigma^0_2}$ ideal which is not countably generated, then there are $A\in\I^+$, $c\in\omega^\omega$ and a bijection $f:A\to T_c$ such that:
\begin{itemize}
    \item $\phi(f^{-1}[\{t^\frown(i):\ i<c(n)\}])>n+1$, for every $t\in T_c\cap\omega^n$;
    \item $f^{-1}[A_x]\in\I\restriction A$, for every $x\in \hat{T}_c$.
\end{itemize}
\end{lemma}

\begin{proof} 
In this proof, we will write that a family $\cA\subseteq\cP(\omega)$ is $\I$-countably generated if there is a countable family $\cB\subseteq\I$ such that for each $A\in\cA$ there is $B\in\cB$ with $A\subseteq B$. Note that a countable union of $\I$-countably generated families is $\I$-countably generated.

Define $\I_k=\{A\in\I:\ \phi(A)\leq k\}$, for $k\in\omega$. Then $\bigcup_{k\in\omega}\I_k=\I$. Since $\I$ is not $\I$-countably generated (otherwise it would be countably generated), there is $k\in\omega$ such that $\I_k$ is not $\I$-countably generated. 

We will inductively construct for each $n\in\omega$ a set $G_n\subseteq\omega^n$, $c(n)\in\omega$ and finite sets $F_s$, for all $s\in G_n$, satisfying the following properties:
\begin{itemize}
    \item[(a)] $G_0=\{\emptyset\}$ and $G_n=\{s\in\omega^n:\ s(0)\in F_\emptyset\text{ and }s(i)\in F_{s\restriction i}\text{ for all }1\leq i<n\}$, if $n\in\omega\setminus\{0\}$;
    \item[(b)] $\phi(F_s)>n+1$, for all $s\in G_n$;
    \item[(c)] $F_s\cap F_t=\emptyset$ and $0\notin F_s$, for all $s,t\in\bigcup_{m\leq n}G_m$, $s\neq t$;
    \item[(d)] $c(n)=|F_s|$, for all $s\in G_n$;
    \item[(e)] $A_{s^\frown(i)}=\{A\in\I_k:\ i\in A\text{ and }s(j)\in A\text{ for all }j<n\}$ is not $\I$-countably generated, for each $s\in G_n$ and each $i\in F_s$.
\end{itemize}

At the first step, put $G_0=\{\emptyset\}$. For each $i\in\omega$, let $A_{(i)}=\{A\in\I_k:\ i\in A\}$. Denote
$$T^\emptyset=\{i\in\omega:\ A_{(i)}\text{ is not }\I\text{-countably generated}\}.$$ 
Observe that $T^\emptyset\in\I^+$. Indeed, otherwise $\cP(T^\emptyset)$ would be $\I$-countably generated (even by one set) and we would get that $\I_k=(\cP(T^\emptyset)\cap\I_k)\cup\bigcup_{i\in\omega\setminus T^\emptyset} A_{(i)}$ is $\I$-countably generated, which is a contradiction. Since $T^\emptyset\in\I^+$, we can find a finite $F_\emptyset\subseteq T^\emptyset$ such that $0\notin F_\emptyset$ and $\phi(F_\emptyset)>1$. Put $c(0)=|F_\emptyset|$. Then items (a), (b), (c), (d) and (e) are met. 

In the $n$th step, if all $c(m)$ and $G_m$, for $m<n$, and all $F_s$, for $s\in \bigcup_{m<n}G_m$, are defined, put $G_n=\{s\in\omega^n:\ s(0)\in F_\emptyset\text{ and }s(i)\in F_{s\restriction i}\text{ for all }1\leq i<n\}$. Note that $G_n$ is finite as each $F_s$, for $s\in \bigcup_{m<n}G_m$, is finite. Enumerate $G_n=\{s_j:\ j<l_n\}$. We will inductively construct finite sets $F'_{s_j}$, for $j<l_n$, satisfying items (b), (c) and (e).

By item (e) (applied to $n-1$, $s_0\restriction (n-1)\in G_{n-1}$ and $s_0(n-1)\in F_{s\restriction (n-1)}$), the set $A_{s_0}=\{A\in\I_k:\ s_0(i)\in A\text{ for all }i<n\}$ is not $\I$-countably generated. For each $i\in\omega$, let $A_{s_0^\frown(i)}=\{A\in A_{s_0}:\ i\in A\}$. Denote
$$T^{s_0}=\{i\in\omega:\ A_{s_0^\frown(i)}\text{ is not }\I\text{-countably generated}\}.$$ 
Observe that $T^{s_0}\in\I^+$. Indeed, otherwise we would get that $A_{s_0}=(\cP(T^{s_0})\cap A_{s_0})\cup\bigcup_{i\in\omega\setminus T^{s_0}} A_{s_0^\frown(i)}$ is $\I$-countably generated, which is a contradiction. Since $T^{s_0}\in\I^+$, also $T^{s_0}\setminus K_{s_0}\in\I^+$, where $K_{s_0}=\{0\}\cup\bigcup_{m<n}\bigcup_{s\in G_m}F_s\in\Fin$. Thus, we can find a finite set $F'_{s_0}\subseteq T^{s_0}\setminus K_{s_0}$ such that $\phi(F'_{s_0})>n+1$. 

Assume that sets $F'_{s_j}$, for $j<p<l_n$, are defined. By item (e), the set $A_{s_p}=\{A\in\I_k:\ s_p(i)\in A\text{ for all }i<n\}$ is not $\I$-countably generated. Similarly as before, for each $i\in\omega$, let $A_{s_p^\frown(i)}=\{A\in A_{s_p}:\ i\in A\}$ and
$$T^{s_p}=\{i\in\omega:\ A_{s_p^\frown(i)}\text{ is not }\I\text{-countably generated}\}.$$ 
Since $T^{s_p}\in\I^+$ (for the same reason as above) and 
$$K_{s_p}=\{0\}\cup\left(\bigcup_{m<n}\bigcup_{s\in G_m}F_s\right)\cup\left(\bigcup_{j<p}F'_{s_j}\right)$$
is finite, $T^{s_p}\setminus K_{s_p}\in\I^+$ and we can find a finite set $F'_{s_p}\subseteq T^{s_p}\setminus K_{s_p}$ such that $\phi(F'_{s_p})>n+1$.

Once all $F'_{s_j}$ are defined, let $c(n)=\max\{|F'_{s_j}|:\ j<l_n\}$. We will define sets $F_{s_j}$, for $j<l_n$, inductively. Let $F_{s_0}$ be any finite subset of $T^{s_0}\setminus K_{s_0}$ of cardinality $c(n)$ containing $F'_{s_0}$ and disjoint with $\bigcup_{0<j<l_n}F'_{s_j}$ (which exists as $T^{s_0}\setminus K_{s_0}\in\I^+$ and $\bigcup_{j<l_n}F'_{s_j}$ is finite). If $F_{s_j}$, for all $j<p<l_n$, are defined, let $F_{s_p}$ be any finite subset of $T^{s_p}\setminus K_{s_p}$ of cardinality $c(n)$ which contains $F'_{s_p}$ and is disjoint with $\bigcup_{j<p}F_{s_j}\cup\bigcup_{p<j<l_n}F'_{s_j}$.

It is easy to verify that all the required conditions are met. This finishes the inductive construction of $c\in\omega^\omega$ and sets $G_n$ and $F_s$.

Define $A=\{0\}\cup\bigcup_{n\in\omega}\bigcup_{s\in G_n}F_s$. Then clearly $A\in\I^+$. 

Now we will define $f:A\to T_c$. Fix any $a\in A$. If $a=0$, put $f(0)=\emptyset$. Otherwise, there are $n\in\omega$ and $s\in G_n$ such that $a\in F_s$. Note that $n$ and $s$ are unique (by item (c)). Let $f(a)\in T_c\cap\omega^{n+1}$ be such that:
\begin{itemize}
    \item $(f(a))(n)<c(n)$ is such that $a$ is the $(f(a))(n)$th element of $F_s$;
    \item $(f(a))(i)<c(i)$ is such that $s(i)$ is the $(f(a))(i)$th element of $F_{s\restriction i}$, for all $i<n$.
\end{itemize}

We need to check that $f$ defined in this way is as needed. It is easy to see that $f$ is an injection.

Fix any $t=(t(0),t(1),\ldots,t(n-1))\in T_c\cap\omega^{n}$. If $t=\emptyset$, then $f(0)=t$. Otherwise, let $s\in G_{n}$ be such that:
\begin{itemize}
    \item $s(0)$ is the $t(0)$th element of $F_\emptyset$;
    \item $s(i)$ is the $t(i)$th element of $F_{s\restriction i}$, for all $i<n$;
\end{itemize}
(such $s$ exists by item (d)). Then $f(s(n-1))=t$, hence $f$ is a surjection. Moreover, observe that $\phi(f^{-1}[\{t^\frown(i):\ i<c(n)\}])>n+1$ as $f^{-1}[\{t^\frown(i):\ i<c(n)\}]=F_s$ and $\phi(F_s)>n+1$ (by item (b)). 

Fix now any $x\in\hat{T}_c$. Let $F\subseteq f^{-1}[A_x]$ be any finite set such that $0\notin F$. We will show that $\phi(F)\leq k$, which will imply that $\phi(f^{-1}[A_x])\leq k+\phi(\{0\})$ and finish the proof. 

There is $n\in\omega$ such that $F\subseteq M=f^{-1}[\{x\restriction i:\ i\leq n+1\}]$. Let $m_0$ be the unique element of $M\cap F_\emptyset$ (which exists as $f^{-1}(x(0))\in F_\emptyset$ and $f^{-1}(x(i))\in A\setminus F_\emptyset$, for all $i>0$) and $m_{i+1}$ be the unique element of $M\cap F_{(m_0,\ldots,m_i)}$, for all $i<n$. Observe that $(m_0,\ldots,m_{n-1})\in G_{n}$ and $m_n\in F_{(m_0,\ldots,m_{n-1})}$. Thus, by item (e), $\{A\in\I_k:\ m_i\in A\text{ for all }i\leq n\}$ is not $\I$-countably generated. In particular, there is some $A\in\I_k$ such that $M=\{m_i:\ i\leq n\}\subseteq A$. Since $A\in\I_k$, we have $\phi(F)\leq\phi(M)\leq\phi(A)\leq k$. This finishes the entire proof.
\end{proof}

\begin{lemma}
\label{lem2Fsigma}
Every non-pathological $\bf{\Sigma^0_2}$ ideal which is not countably generated, is not Egorov.
\end{lemma}

\begin{proof}
This proof is based on ideas from \cite[Theorem 3.4]{Nikodem}.

Let $\I=\Fin(\phi)$ be a non-pathological $\bf{\Sigma^0_2}$ ideal which is not countably generated. By Lemma \ref{lem1Fsigma}, there are $A\in\I^+$, $c\in\omega^\omega$ and a bijection $f:A\to T_c$ such that:
\begin{itemize}
    \item $\phi(G_t)>n+1$, where $G_t=f^{-1}[\{t^\frown(i):\ i<c(n)\}]$, for every $n\in\omega$ and $t\in T_c\cap\omega^n$;
    \item $f^{-1}[A_x]\in\I\restriction A$, for every $x\in \hat{T}_c$.
\end{itemize}
Since $\phi$ is non-pathological, for each $n\in\omega$ and $t\in T_c\cap\omega^n$ there is a measure $\mu_t$ on $\omega$ such that $\mu_t\leq \phi$ and $\mu_t(G_t)=n+1$.

We will inductively define open intervals $I_t\subseteq[0,1]$, for all $t\in T_c$. Start with any family $\{I_{(j)}:\ j<c(0)\}$ of open pairwise disjoint subintervals of $[0,1]$ such that the length of $I_{(j)}$ is $\mu_\emptyset(\{f^{-1}((j))\})$, for each $j<c(0)$. Note that this is possible, since $\sum_{j<c(0)}\mu_\emptyset(\{f^{-1}((j))\})=\mu_\emptyset(G_\emptyset)=1$. If $I_t$, for all $t\in T_c\cap\omega^n$, are already defined, for each $t\in T_c\cap\omega^n$ let $\{I_{t^\frown(j)}:\ j<c(n)\}$ be any family of open pairwise disjoint subintervals of $I_t$ such that:
$$\lambda(I_{t^\frown(j)})=\lambda(I_t)\frac{\mu_t(\{f^{-1}(t^\frown(j))\})}{n+1},$$ 
for each $j<c(n)$. Note that $\sum_{j<c(n)}\lambda(I_{t^\frown(j)})=\lambda(I_t)$.

Define a sequence $(g_n)_{n\in\omega}$ of real-valued functions on $[0,1]$ by:
\begin{itemize}
    \item if $n\in A$ then $g_n$ is the characteristic function of $I_{f(n)}$;
    \item if $n\notin A$ then $g_n$ is constantly equal to $0$.
\end{itemize}

Observe that $(g_n)_{n\in\omega}$ is $\I$-pointwise convergent to $0$, as for each $x\in[0,1]$ either $x$ is an endpoint of some $I_t$ (hence, an endpoint of some $I_{s}$ with $s\in T_c\cap\omega^{|t|+n}$ and $t\subseteq  s$, for every $n\in\omega$) and $\{n\in\omega:\ g_n(x)=1\}$ is finite, or there is $\hat{x}\in \hat{T}_c$ such that $\{n\in\omega:\ g_n(x)=1\}=f^{-1}[A_{\hat{x}}]\in\I$.

Let $M\subseteq[0,1]$ be measurable with $\lambda(M)=\alpha>0$. Now we will show that $(g_n\restriction M)_{n\in\omega}$ does not converge $\I$-uniformly to $f$, i.e., that $\{n\in\omega:\ \exists_{x\in M}\ g_n(x)=1\}\notin\I$. Actually, we need to prove that $\phi(\{n\in\omega:\ \exists_{x\in M}\ g_n(x)=1\})>m$, for every $m\in\omega$. 

Fix $m\in\omega$ and let $k\in\omega$ be such that $(k+1)\alpha>m$. Observe that there is $t\in T_c\cap\omega^k$ such that $\lambda(M\cap I_t)\geq \alpha \lambda(I_t)$. Then $\sum_{j\in J} \lambda(I_{t^\frown(j)})\geq \lambda(M\cap I_t)\geq \alpha \lambda(I_t)$, where $J=\{j<c(n):\ M\cap I_{t^\frown(j)}\neq\emptyset\}$. Thus, 
$$\phi(\{n\in\omega:\ \exists_{x\in M}\ g_n(x)=1\})\geq \mu_t(\{n\in G_t:\ \exists_{x\in M}\ g_n(x)=1\})=$$
$$=\mu_t(\{f^{-1}(t^\frown(j)):\ j\in J\})=\sum_{j\in J}\mu_t(\{f^{-1}(t^\frown(j))\})=$$
$$=\frac{k+1}{\lambda(I_t)}\sum_{j\in J} \lambda(I_{t^\frown(j)})\geq (k+1)\alpha>m.$$
This finishes the proof.
\end{proof}

\begin{theorem}
\label{thmFsigma}
Let $\I$ be a non-pathological $\bf{\Sigma^0_2}$ ideal. The following are equivalent:
\begin{itemize}
    \item[(a)] $\I$ is Egorov;
    \item[(b)] $\I$ is countably generated;
    \item[(c)] $\I$ is isomorphic to $\Fin$, $\Fin\oplus\cP(\omega)$ or $\Fin\otimes\{\emptyset\}$.
\end{itemize}
\end{theorem}

\begin{proof}
(a)$\implies$(b): It follows from Lemma \ref{lem2Fsigma}.

(b)$\implies$(a): This is \cite[Corollary 2.9]{KadetsLeonov}.

(c)$\implies$(b): It is obvious that each of the ideals $\Fin$, $\Fin\oplus\cP(\omega)$ and $\Fin\otimes\{\emptyset\}$ is countably generated and one can easily check that an isomorphic copy of a countably generated ideal is countably generated.

(b)$\implies$(c): Let $\{A_n:\ n\in\omega\}\subseteq\I$ be the countable family generating $\I$. Define $B_n=A_n\setminus \bigcup_{i<n}A_i$, for each $n\in\omega$, and observe that $B_n\cap B_m=\emptyset$, for all $n\neq m$, and $\{B_n:\ n\in\omega\}\subseteq\I$ is also a countable family generating $\I$. Moreover, $\bigcup_{n\in\omega}B_n=\omega$, since $\Fin\subseteq\I$. 

Define $T=\{n\in\omega:\ B_n\in\Fin^+\}$. If $T=\emptyset$, then $\I$ is isomorphic to $\Fin$. If $T\neq\emptyset$, but $T\in\Fin$, then $\I$ is isomorphic to $\Fin\oplus\cP(\omega)$ (as witnessed by any bijection $f:\omega\to(2\times\omega)$ such that $f[\bigcup_{n\in T}B_n]=\{1\}\times\omega$). If $T\in\Fin^+$, then enumerate $T=\{t_n:\ n\in\omega\}$. If $\omega\setminus\bigcup_{n\in T}B_n$ is finite, let $f:\omega\to\omega^2$ be any bijection such that $f[B_{t_0}\cup(\omega\setminus\bigcup_{n\in T}B_n)]=\{0\}\times\omega$ and $f[B_{t_n}]=\{n\}\times\omega$ for $n\in\omega\setminus\{0\}$. If $\omega\setminus\bigcup_{n\in T}B_n$ is infinite, let $f:\omega\to\omega^2$ be any bijection such that $f[\omega\setminus\bigcup_{n\in T}B_n]=\omega\times\{0\}$ and $f[B_{t_n}]=\{n\}\times(\omega\setminus\{0\})$ for $n\in\omega$. In both cases, $f$ witnesses that $\I$ is isomorphic to $\Fin\otimes\{\emptyset\}$.
\end{proof}

\begin{corollary}
\label{corFsigma}
If $\I$ is a tall non-pathological $\bf{\Sigma^0_2}$ ideal, then it is not Egorov.
\end{corollary}

\begin{proof}
It follows from Theorem \ref{thmFsigma}, as $\Fin$, $\Fin\oplus\cP(\omega)$ and $\Fin\otimes\{\emptyset\}$ all are not tall.
\end{proof}

\begin{corollary}
\label{GenEgorov-cor1}
Let $\I$ be a non-pathological $\bf{\Sigma^0_2}$ ideal. The following are equivalent:
\begin{itemize}
    \item[(a)] $\I$ is Egorov in the sense of Kadets and Leonov, i.e., for every finite measure space $(X,\cM,\mu)$, if $\eta>0$ and $(f_n)_{n\in\omega}$ is a sequence of measurable real-valued functions defined on $X$ which is $\I$-pointwise convergent to a measurable function $f:X\to\mathbb{R}$, then there is $M\in\cM$ such that $\mu(X\setminus M)\leq\eta$ and $(f_n\restriction M)_{n\in\omega}$ is $\I$-uniformly convergent to $f\restriction M$;
    \item[(b)] under $\non(\cN)<\bb$, $\I$ is generalized Egorov in the sense of Korch and Repick\'{y}, i.e., if $\eta>0$ and $(f_n)_{n\in\omega}$ is a sequence of (non-necessarily measurable) real-valued functions defined on $[0,1]$ which is $\I$-pointwise convergent to a (non-necessarily measurable) function $f:[0,1]\to\mathbb{R}$, then there is $M\subseteq[0,1]$ such that $\lambda^\star(M)\geq 1-\eta$ (here $\lambda^\star$ is the Lebesgue's outer measure) and $(f_n\restriction M)_{n\in\omega}$ is $\I$-uniformly convergent to $f\restriction M$;
    \item[(c)] $\I$ is countably generated;
    \item[(d)] $\I$ is isomorphic to $\Fin$, $\Fin\oplus\cP(\omega)$ or $\Fin\otimes\{\emptyset\}$.
\end{itemize}
\end{corollary}

\begin{proof}
The equivalence of (c) and (d) is due to Theorem \ref{thmFsigma}. The implication (c)$\implies$(a) is proved in \cite[Corollary 2.9]{KadetsLeonov} while (c)$\implies$(b) is shown in \cite[Corollary 10]{Korch}. The implications (a)$\implies$(c) and (b)$\implies$(c) follow from Theorem \ref{thmFsigma}, as every non-Egorov (in our sense) ideal is not Egorov in the sense of Kadets and Leonov and not generalized Egorov in the sense of Korch and Repick\'{y}. 
\end{proof}

\section{Rudin-Keisler order and pathological \texorpdfstring{$\bf{\Sigma^0_2}$}{F } ideals}
\label{sec:RK}

In the previous Section we have seen that, up to isomorphism, there are only three non-pathological $\bf{\Sigma^0_2}$ ideals. In this section we deal with pathological $\bf{\Sigma^0_2}$ ideals. There are very few examples of such ideals. Probably the best known is Mazur's ideal $\mathcal{M}$ introduced in \cite[Lemma 1.8 and Theorem 1.9]{Mazur} (for definition see the paragraph before Proposition \ref{S+M}; see also \cite[Subsection 3.2.1]{MezaPat}). Also, it is unknown whether the Solecki's ideal $\mathcal{S}$ (for definition see the paragraph before Proposition \ref{S+M}) is patological (see \cite[Question 3.13]{MezaPat}). We do not answer this question, however we show that both $\mathcal{M}$ and $\mathcal{S}$ are non-Egorov. In the proof we apply the Rudin-Keisler order on ideals. Next result reveals its connection with Egorov ideals, hence, justifies studies of Rudin-Keisler order among $\bf{\Sigma^0_2}$ ideals.

\begin{proposition}
\label{RK-main}
If $\J$ is Egorov and $\I\leq_{RK}\J\restriction A$, for some $A\in\J^+$, then $\I$ is Egorov.
\end{proposition}

\begin{proof}
By Proposition \ref{restrictions}, it suffices to show that if $\J$ is Egorov and $\I\leq_{RK}\J$, then $\I$ is Egorov.

Let $\eta>0$ and $(f_n)_{n\in\omega}$ be any sequence of measurable real-valued functions defined on $[0,1]$, which is $\I$-pointwise convergent to a measurable function $f$. Denote by $h:\omega\to\omega$ the function witnessing $\I\leq_{RK}\J$ and define $g_i=f_{h(i)}$, for all $i\in\omega$. Observe that $(g_i)_{i\in\omega}$ is $\J$-pointwise convergent to $f$, as for each $x\in[0,1]$ and $\varepsilon>0$ we have:
$$\{i\in\omega:\ |g_i(x)-f(x)|\geq\varepsilon\}\subseteq h^{-1}[\{n\in\omega:\ |f_n(x)-f(x)|\geq\varepsilon\}]\in\J.$$

Thus, since $\J$ is Egorov, there is a measurable $M\subseteq[0,1]$ such that $\lambda([0,1]\setminus M)>\eta$ and $(g_i\restriction M)_{i\in\omega}$ is $\J$-uniformly convergent to $f$. To finish the proof, we will show that $(f_n\restriction M)_{n\in\omega}$ is $\I$-uniformly convergent to $f$. Fix $\varepsilon>0$ and note that:
$$h^{-1}\left[\left\{n\in\omega:\ \exists_{x\in M}\ |f_n(x)-f(x)|\geq\varepsilon\right\}\right]\subseteq\left\{i\in\omega:\ \exists_{x\in M}\ |g_i(x)-f(x)|\geq\varepsilon\right\}\in\J.$$
Thus, $\left\{n\in\omega:\ \exists_{x\in M}\ |f_n(x)-f(x)|\geq\varepsilon\right\}\in\I$. 
\end{proof}

Recall definitions of $\bf{\Sigma^0_2}$ ideals that we will work with in this section:
\begin{itemize} 
\item $\I_b$ is an ideal on $2^{<\omega}$ generated by the family $\{A_x:\ x\in 2^\omega\}$, where $A_x=\{x\restriction n:\ n\in\omega\}$. It is $\bf{\Sigma^0_2}$, not tall and not a P-ideal (\cite[Subsection 1.4]{Mazur}; see also \cite[Section 3]{Zakrzewski}).
\item $\ED_\Fin=\{A\subseteq\Delta:\ \exists_{k\in\omega}\ \forall_{n\in\omega}\ |A_{(n)}|\leq k\}$, where $\Delta=\{(i,j)\in\omega^2:\ i\geq j\}$, is a tall $\bf{\Sigma^0_2}$ non-P-ideal (\cite[Section 3.2]{Hrusak}).
\item Denote by $\Omega$ the family of all clopen subsets of $2^\omega$ of measure $\frac{1}{2}$ (with respect to the standard Haar measure on $2^\omega$). Then the Solecki's ideal is the ideal on $\Omega$ generated by sets of the form $B_x=\{A\in\Omega:\ x\in A\}$, for $x\in 2^\omega$. This is a tall $\bf{\Sigma^0_2}$ non-P-ideal (\cite[Section 3.6]{Hrusak}).
\item Mazur's ideal $\mathcal{M}$ is an ideal on the disjoint sum $\sum_{n\in\omega}(2n)^{n}$ (here $(2n)^{n}$ is the set of all functions from $n$ to $2n$) consisting of those $A\subseteq \sum_{n\in\omega}(2n)^{n}$ with $\sup_{n\in\omega}\phi_n(A_{(n)})<\infty$, where
$$\phi_n(B)=\min\left\{|F|:\ F\subseteq 2n\text{ and }B\subseteq\bigcup_{i\in F}M^n_i\right\}$$
for all $B\subseteq (2n)^{n}$ and $M^n_i=\{f\in (2n)^{n}:\ i\notin f[n]\}$. $\mathcal{M}$ is a tall $\bf{\Sigma^0_2}$ non-P-ideal (\cite[Lemma 1.8 and Theorem 1.9]{Mazur}; see also \cite[Subsection 3.2.1]{MezaPat}).
\end{itemize}

\begin{proposition}\
\label{S+M}
\begin{itemize}
    \item[(a)] $\I_b\leq_{RK}\mathcal{S}$;
    \item[(b)] $\ED_\Fin\leq_{RB}\mathcal{M}$;
    \item[(c)] $\I_b\leq_{RK}\ED_\Fin$.
\end{itemize}
\end{proposition}

\begin{proof}
(a): In this proof, for $s,t\in 2^{<\omega}$ we write $s\preceq t$ if there is $m\leq\min\{|s|,|t|\}$ such that $s(i)=t(i)$, for all $i<m$, and $s(m)<t(m)$. Given a clopen set $C\subseteq 2^\omega$ of measure $\frac{1}{2}$ (i.e., $C\in\Omega$), there are $n_C\in\omega$ and $s^C_i\in 2^{<\omega}$, for $i\leq n_C$, such that $C=\bigcup_{i\leq n_C}[s^C_i]$, where $[s]=\{x\in 2^\omega:\ s\subseteq x\}$. Without loss of generality we may assume that $s^C_i$, for $i\leq n_C$, are pairwise $\subseteq$-incomparable (it suffices to consider only those $s^C_i$ such that $s^C_j\not\subseteq s^C_i$ for all $j\leq n_C$, $j\neq i$, and observe that $\bigcup_{i\leq n_C}[s^C_i]=\bigcup\{[s^C_i]:\ i\leq n_C\text{ and }\forall_{j\leq n_C,j\neq i}\ s^C_j\not\subseteq s^C_i\}$). 

For each $C\in\Omega$ let $i_C\leq n_C$ be such that $s^C_{i_C}\preceq s^C_{i}$, for all $i\leq n_C$, and define $f:\Omega\to 2^{<\omega}$ by $f(C)=g(s^C_{i_C})$, for all $C\in\Omega$, where $g:(2^{<\omega}\setminus \{\emptyset\})\to 2^{<\omega}$ is given by $g(s_0,\ldots,s_k)=(s_1,\ldots,s_k)$, for all $(s_0,\ldots,s_k)\in 2^{<\omega}\setminus \{\emptyset\}$. Note that if $C\in\Omega$ is not equal to $[(1)]$, then $s^C_{i_C}$ starts with $0$. 

We claim that $f$ witnesses $\I_b\leq_{RK}\mathcal{S}$. First we show that $A\in\I_b$ implies $f^{-1}[A]\in\mathcal{S}$. Actually, it suffices to show that $f^{-1}[A_x]\in\mathcal{S}$, for all $x\in 2^\omega$. Observe that $f^{-1}[A_x]\subseteq \{[(1)]\}\cup B_{(0)^\frown x}\in\mathcal{S}$, as:
$$C\in f^{-1}[A_x]\ \implies\ (C=[(1)]\text{ or }(0)^\frown x\in[s^C_{i_C}])\ \implies\ (C=[(1)]\text{ or }(0)^\frown x\in C).$$

Assume now that $A\notin\I_b$ and fix any $k\in\omega$. We will show that $f^{-1}[A]$ cannot be covered by $k$ many sets of the form $B_x$. Let $x_0,\ldots,x_{k-1}\in 2^\omega$ be arbitrary. Since $A\notin\I_b$, there is an $\subseteq$-antichain $\{s_0,\ldots,s_k\}\subseteq A$. Then there is $n\leq k$ such that $s_n\notin\bigcup_{i<k} A_{\bar{x}_i}$, where for $x=(x(0),x(1),x(2),\ldots)\in 2^\omega$ we denote $\bar{x}=(x(1),x(2),\ldots)$. Thus, $x_i\notin[(0)^\frown s_n]$, for all $i<k$. Denote by $\mu$ the standard Haar measure on $2^\omega$. Note that $\mu([(0)^\frown s_n])=\frac{1}{2^m}$ for some $m\in\omega\setminus \{0\}$. Let $l>m$ be such that $k\leq 2^{l-m}$. Observe that the above implies that $\frac{1}{2}-\frac{k}{2^{l}}\geq\frac{1}{2}-\mu([(0)^\frown s_n])$. Note that among sequences $t\in 2^{l-1}$ at most $k$ of them is such that there is $i<k$ with $x_i\in[(1)^\frown t]$. Hence, denoting $p=2^{l-1}-2^{l-m}$, we see that $p\leq 2^{l-1}-k$, so we can find pairwise distinct $t_0,\ldots,t_{p-1}\in 2^{l-1}$ such that $x_i\notin[(1)^\frown t_j]$, for all $i<k$ and all $j<p$. 

Define $C=[(0)^\frown s_n]\cup\bigcup_{j<p}[(1)^\frown t_j]$ and observe that:
$$\mu\left(C\right)=\mu\left([(0)^\frown s_n]\right)+\sum_{j<p}\mu\left([(1)^\frown t_j]\right)=\frac{1}{2^m}+\left(2^{l-1}-2^{l-m}\right)\frac{1}{2^l}=\frac{1}{2},$$ 
so $C\in\Omega$. Moreover, $C\notin\bigcup_{i<k}B_{x_i}$, but $C\in f^{-1}[\{s_n\}]\subseteq f^{-1}[\{s_0,\ldots,s_k\}]\subseteq f^{-1}[A]$. Hence, $f^{-1}[A]\not\subseteq \bigcup_{i<k}B_{x_i}$.

(b): For each $n\in\omega$ and $i\leq n$ denote $S_i^n=M^n_i\setminus \bigcup_{j<i}M^n_j$. Observe that $\bigcup_{i\leq n}S_i^n=(2n)^n$. Moreover, $\phi_n(\bigcup_{i\in F}S_i^n)=|F|$, for each $F\subseteq (n+1)$. Indeed, $\phi_n(\bigcup_{i\in F}S_i^n)\leq|F|$ is obvious as $\bigcup_{i\in F}S_i^n\subseteq \bigcup_{i\in F}M_i^n$. On the other hand, if $|F|\leq 1$ then $\phi_n(\bigcup_{i\in F}S_i^n)\geq|F|$ is obvious and assuming that $|F|>1$, if $i_{|F|-2}<\ldots<i_0<2n$ are arbitrary, then let $m=\min\left(F\setminus\{i_0,\ldots,i_{|F|-2}\}\right)$. Observe that: 
$$|\{j\leq |F|-2:\ i_j\geq m\}|=|F|-1-|\{j\leq |F|-2:\ i_j<m\}|=$$
$$=|F|-1-|F\cap m|=|F\setminus (m+1)|\leq (n+1)-(m+1)=n-m.$$
Let $g:n\to 2n$ be the function given by:
$$g(j)=\begin{cases}
    j, & \text{if }j<m, \\
    i_{j-m}, & \text{if }m\leq j\leq m+|F|-2, \\
    i_0, & \text{if } m+|F|-2<j<n.
\end{cases}$$
Then clearly $g\in S_k^m\subseteq \bigcup_{i\in F}S_i^n$, but $g\notin M^n_{i_j}$ for all $j\leq |F|-2$ (as there are at most $n-m$ many $j\leq |F|-2$ with $i_j\geq m$), which proves $\bigcup_{i\in F}S_i^n\not\subseteq\bigcup_{j\leq |F|-2}M^n_{i_j}$. Hence, we can conclude that $\phi_n(\bigcup_{i\in F}S_i^n)\geq|F|$. 

Let $f:\sum_{n\in\omega}(2n)^n\to\Delta$ be given by $f[S^n_i]=\{(n,i)\}$. We claim that $f$ witnesses $\ED_\Fin\leq_{RB}\mathcal{M}$. 

It is obvious that $f$ is finite-to-one. Moreover, if $A\in\ED_\Fin$, then there is $k\in\omega$ such that $|A_{(n)}|\leq k$ for all $n\in\omega$ and we have:
$$\left(f^{-1}[A]\right)_{(n)}\subseteq \bigcup_{i\in A_{(n)}}M^n_i,$$
hence $\phi_n(\left(f^{-1}[A]\right)_{(n)})\leq k$ and $f^{-1}[A]\in\mathcal{M}$. Assume now that $A\notin\ED_\Fin$ and let $k\in\omega$ be arbitrary. We need to show that $\sup_{n\in\omega}\phi_n(\left(f^{-1}[A]\right)_{(n)})\geq k$. Since $A\notin\ED_\Fin$, there is $n\in\omega$ with $|A_{(n)}|\geq k$. Then 
$$\phi_n\left(\left(f^{-1}[A]\right)_{(n)}\right)=\phi_n\left(\bigcup_{i\in A_{(n)}}S^n_i\right)=|A_{(n)}|\geq k.$$

(c): Let $\{s_n:\ n\in\omega\}$ be an enumeration of $2^{<\omega}$. Inductively, for each $n\in\omega$, pick $A_{n}\subseteq\omega$ such that:
\begin{itemize}
    \item[(i)] $(\omega\setminus\bigcup_{k<n}A_{k})\cap A_{n}$ and $(\omega\setminus\bigcup_{k<n}A_{k})\setminus A_{n}$ both are infinite;
    \item[(ii)] if $s_k\subseteq s_n$ for some $k<n$, then $A_{n}\cap A_{k}=\emptyset$;
    \item[(iii)] if $T\subseteq n$ is non-empty and $\{s_n\}\cup\{s_k:\ k\in T\}$ is an $\subseteq$-antichain, then $(\bigcap_{k\in T}A_{k}\setminus\bigcup_{k\in n\setminus T}A_{k})\cap A_{n}$ and $(\bigcap_{k\in T}A_{k}\setminus\bigcup_{k\in n\setminus T}A_{k})\setminus A_{n}$ both are infinite.
\end{itemize}
This is possible as after defining $A_n$ for some $n\in\omega$ we have $\omega\setminus\bigcup_{k\leq n}A_{k}\in\Fin^+$ (by (i)) and $\bigcap_{k\in T}A_{k}\setminus\bigcup_{k\in (n+1)\setminus T}A_{k}\in\Fin^+$, for each non-empty $T\subseteq (n+1)$ such that $\{s_k:\ k\in T\}$ is an $\subseteq$-antichain (by (ii) and (iii)).

Define $A=\bigcup_{n\in\omega}(A_{n}\setminus n)\times\{n\}$. Then $A\subseteq \Delta$. Let $f:A\to 2^{<\omega}$ be given by $f[(A_n\setminus n)\times\{n\}]=\{s_n\}$. 

Observe that if $B\notin\I_b$, then $f^{-1}[B]\notin\ED_\Fin$. Indeed, for each $k\in\omega$ we can find an $\subseteq$-antichain $\{s_{m_i}:\ i\leq k\}\subseteq B$. Then $\bigcap_{i\leq k}A_{m_i}$ is infinite, so there is $x\in\bigcap_{i\leq k}A_{m_i}\setminus \max_{i\leq k}m_i$ and we have:
$$\{x\}\times\{m_i:\ i\leq k\}\subseteq \bigcup_{i\leq k}(A_{m_i}\setminus m_i)\times\{m_i\}=f^{-1}[\{s_{m_i}:\ i\leq k\}]\subseteq f^{-1}[B].$$

In particular, $A\notin\ED_\Fin$, since $A=f^{-1}[2^{<\omega}]$ and $2^{<\omega}\notin\I_b$. Since $\ED_\Fin\restriction A$ and $\ED_\Fin$ are isomorphic (by \cite[Example 2.4]{homogeneous}), it is enough to show that $f$ witnesses $\I_b\leq_{RK}\ED_\Fin\restriction A$. We already know that $B\notin\I_b$ implies $f^{-1}[B]\notin\ED_\Fin$, so it remains to prove that $B\in\I_b$ implies $f^{-1}[B]\in\ED_\Fin$. Actually, it is enough to show that $f^{-1}[A_x]\in\ED_\Fin$, for all $x\in 2^\omega$. 

Fix any $x\in 2^\omega$ and observe that if $s_n=x\restriction i$ and $s_m=x\restriction j$, for some $i,j\in\omega$ with $i\neq j$, then $A_n\cap A_m=\emptyset$ (by (ii)). Hence, $|f^{-1}[A_x]\cap(\{k\}\times\omega)|\leq 1$, for all $k\in\omega$ an we can conclude that $f^{-1}[A_x]\in\ED_\Fin$.
\end{proof}

\begin{corollary}\
$\mathcal{S}$ and $\mathcal{M}$ are non-Egorov.
\end{corollary}

\begin{proof}
By Theorem \ref{thmFsigma}, $\I_b$ is non-Egorov (it is easy to see that it is not countably generated and it is equal to $\Fin(\phi_{\I_b})$, where 
$$\phi_{\I_b}(A)=\sup\{|B|:\ B\subseteq A\text{ is a }\subseteq\text{-antichain}\},$$
for all $A\subseteq 2^{<\omega}$, is a non-pathological lsc submeasure on $2^{<\omega}$). Thus, it suffices to apply Propositions \ref{RK-main} and \ref{S+M}.
\end{proof}

\begin{remark}
Actually, for the above proof we do not need item (c) of Proposition \ref{S+M}, as $\ED_\Fin$ is also non-pathological: $\ED_\Fin=\Fin(\phi_{\ED_\Fin})$, where 
$$\phi_{\ED_\Fin}(A)=\sup_{n\in\omega}|A\cap (\{n\}\times\omega)|,$$
for all $A\subseteq\Delta$, is a non-pathological lsc submeasure on $\Delta$.
\end{remark}

Lemma \ref{lem1Fsigma} and Proposition \ref{S+M} suggest that perhaps for every $\bf{\Sigma^0_2}$ ideal $\I$ we have:
$$\exists_{A\in\I^+}\ \I_b\leq_{RK}\I\restriction A\ \Longleftrightarrow\ \I\text{ is not Egorov.}$$
Next example shows that this is false even for summable ideals, as for instance $\I_{1/n}$ is not Egorov (by Corollary \ref{corFsigma} or \cite[Corollary 3.2 and Example 3.3]{Nikodem}).

\begin{example}
$\I_b\not\leq_{RK}\I_{c_n}\restriction A$, for every summable ideal $\I_{c_n}$ and every $A\notin\I_{c_n}$.
\end{example}

\begin{proof}
It suffices to show that $\I_b\not\leq_{RK}\I_{c_n}$, as a restriction of a summable ideal is still a summable ideal.

Let $f:\omega\to 2^{<\omega}$. Denote by $\bf{0}_n$ the sequence consisting of $n$ zeros and by $\bf{1}$ the sequence $(1,1,1,\ldots)\in 2^\omega$. 

If $f^{-1}[\{\bf{0}_k^\frown (\bf{1}\restriction n):\ n\in\omega\}]\notin\I_{c_n}$, for some $k\in\omega$, then we are done as $\{\bf{0}_k^\frown (\bf{1}\restriction n):\ n\in\omega\}\subseteq A_{\bf{0}_k^\frown\bf{1}}\in\I_b$. If $B_k=f^{-1}[\{\bf{0}_k^\frown (\bf{1}\restriction n):\ n\in\omega\}]\in\I_{c_n}$, for all $k$, then observe that:
$$\sum_{i\in B_k}c_i=\sum_{n\in\omega}\ \sum_{i\in f^{-1}[\{\bf{0}_k^\frown (\bf{1}\restriction n)\}]}c_i.$$
Thus, for each $k\in\omega$ there is $n_k\in\omega$, $n_k>0$ such that:
$$\sum_{i\in f^{-1}[\{\bf{0}_k^\frown (\bf{1}\restriction n_k)\}]}c_i<\frac{1}{2^k}.$$

Consider the set $C=\{\bf{0}_k^\frown (\bf{1}\restriction n_k):\ k\in\omega\}$. Clearly, $C\notin\I_b$, but $f^{-1}[C]\in\I_{c_n}$, as:
$$\sum_{i\in f^{-1}[C]}c_i=\sum_{k\in\omega}\sum_{i\in f^{-1}[\{\bf{0}_k^\frown (\bf{1}\restriction n_k)\}]}c_i<\sum_{k\in\omega}\frac{1}{2^k}<\infty.$$
\end{proof}

\section{Products of ideals}
\label{sec:products}

Results from this section are closely related to \cite[Theorem 3.2]{Repicky}. However, \cite[Theorem 3.2]{Repicky} studies ideals with the property (M') (for definition see \cite{Repicky}). This property implies that $\I$ is Egorov, but we do not know if they are equivalent. Moreover, many our results give characterizations when a product of ideals is Egorov, while \cite[Theorem 3.2]{Repicky} gives only sufficient conditions.

The following result will be a convenient tool for our considerations.

\begin{proposition}[Folklore]
\label{RK-folklore}
Let $\I$, $\J$ and $\I_n$, for $n\in\omega$, be ideals on $\omega$. Then:
\begin{itemize}
    \item[(a)] $\I\leq_{RK}\I\otimes\{\emptyset\}$;
    \item[(b)] $\I\leq_{RK}\sum^\I_{n\in\omega}\I_n$;
    \item[(c)] $\J\leq_{RK}\I\otimes\J$;
    \item[(d)] $\bigcap_{n\in\omega}\I_n\leq_{RK}\sum_{n\in\omega}\I_n$.
\end{itemize}
\end{proposition}

\begin{proof}
The function $f:\omega^2\to\omega$ given by $f(i,j)=i$ is a witness for items (a) and (b). The function $f:\omega^2\to\omega$ given by $f(i,j)=j$ is a witness for items (c) and (d).
\end{proof}

\subsection{Sums}

\begin{proposition}[{\cite[Proposition 2.4]{KadetsLeonov}}]
\label{sums}
Let $\I$ and $\J$ be ideals. Then $\I\oplus\J$ is Egorov if and only if $\I$ and $\J$ both are Egorov.
\end{proposition}

\begin{proposition}
Let $\I$ be an ideal. Then $\I\oplus\cP(\omega)$ is Egorov if and only if $\I$ is Egorov.
\end{proposition}

\begin{proof}
The impication $(\implies)$ follows from Proposition \ref{restrictions} and the other implication is straightforward.
\end{proof}

\subsection{Products of the form \texorpdfstring{$\I\otimes\{\emptyset\}$}{}}

\begin{proposition}
\label{Itimesempty}
Let $\I$ be an ideal on $\omega$. Then $\I\otimes\{\emptyset\}$ is an Egorov ideal if and only if $\I$ is an Egorov ideal. 
\end{proposition}

\begin{proof}
If $\I$ is not Egorov, then $\I\otimes\{\emptyset\}$ is not Egorov by Propositions \ref{RK-main} and \ref{RK-folklore}(a).

Assume that $\I$ is Egorov. We will apply Theorem \ref{characteristic}. Let $\eta>0$ and $(X_{i,j})_{(i,j)\in\omega^2}$ be any sequence of measurable subsets of $[0,1]$ such that $\chi_{X_{i,j}}$ is $\I\otimes\{\emptyset\}$-pointwise convergent to $0$. Define $X_i=\bigcup_{j\in\omega}X_{i,j}$, for all $i\in\omega$. Note that each $X_i$ is a measurable subset of $[0,1]$ and $(X_i)_{i\in\omega}$ is $\I$-pointwise convergent to $0$. Since $\I$ is Egorov, there is a measurable set $M\subseteq[0,1]$ such that $\lambda(M)\leq\eta$ and $A=\{n\in\omega:\ X_n\not\subseteq M\}\in\I$. To finish the proof, observe that $X_{i,j}\subseteq X_i\subseteq M$ for all $(i,j)\in\omega^2\setminus (A\times\omega)$ and $A\times\omega\in\I\otimes\{\emptyset\}$.
\end{proof}

\subsection{Products of the form \texorpdfstring{$\sum_{n\in\omega} \I_n$}{}}

\begin{proposition}
\label{prop:Egorov-sumIn}
Let $(\I_n)_{n\in\omega}$ be a sequence of ideals on $\omega$. Then $\sum_{n\in\omega} \I_n$ is an Egorov ideal if and only if each $\I_n$ is an Egorov ideal. 
\end{proposition}

\begin{proof}
Let $\I=\sum_{n\in\omega} \I_n$. If some $\I_n$ is not an Egorov ideal, then $\I$ is not Egorov by Proposition \ref{restrictions}, as $\I\restriction(\{n\}\times\omega)$ is isomorphic to $\I_n$. Now we will show that if all $\I_n$ are Egorov, then so is $\I$.

Fix a sequence $(f_{(k,n)})_{(k,n)\in\omega^2}$ of measurable real-valued functions defined on $[0,1]$ which is $\I$-pointwise convergent to a measurable function $f:[0,1]\to\R$. Let $\eta>0$. For each $k\in\omega$, since $(f_{(k,n)})_{n\in\omega}$ is $\I_k$-pointwise convergent to $f$, there is $A_k\subseteq[0,1]$ such that $\lambda([0,1]\setminus A_k)<\frac{\eta}{2^{k+1}}$ and $(f_{(k,n)}\restriction A_k)_{n\in\omega}$ is $\I_k$-uniformly convergent to $f\restriction A_k$. Define $A=\bigcap_{k\in\omega} A_k$. Then $\lambda([0,1]\setminus A)=\lambda(\bigcup_{k\in\omega} [0,1]\setminus A_k)\leq\sum_{k\in\omega} \frac{\eta}{2^{k+1}}=\eta$. Moreover, for each $\varepsilon>0$ and $k\in\omega$ we have:
$$\left\{n\in\omega:\ \exists_{x\in A}\ |f_{(k,n)}(x)-f(x)|\geq\varepsilon\right\}\subseteq\left\{n\in\omega:\ \exists_{x\in A_k}\ |f_{(k,n)}(x)-f(x)|\geq\varepsilon\right\}\in\I_k.$$
Consequently, $\left\{(k,n)\in\omega^2:\ \exists_{x\in A}\ |f_{(k,n)}(x)-f(x)|\geq\varepsilon\right\}\in\I$. Thus, $(f_{(k,n)}\restriction A)_{(k,n)\in\omega^2}$ is $\I$-uniformly convergent to $f\restriction A$. 
\end{proof}

\subsection{Countable intersections}

\begin{proposition}
\label{prop:Egorov-cap}
Let $(\I_n)_{n\in\omega}$ be a sequence of ideals on $\omega$. If all $\I_n$ are Egorov, then $\bigcap_{n\in\omega}\I_n$ is an Egorov ideal.
\end{proposition}

\begin{proof}
It is easy to see that an intersection of any family of ideals on $\omega$ is an ideal on $\omega$. The rest follows from Propositions \ref{RK-folklore}(d) and \ref{prop:Egorov-sumIn}.
\end{proof}

\begin{remark}\
\begin{itemize}
    \item It can happen that an intersection of two non-Egorov ideals is Egorov. It suffices to consider ideals $\I=\Fin\oplus\I_{1/n}$ and $\J=\I_{1/n}\oplus\Fin$. Both of them are not Egorov by Proposition \ref{restrictions}, as $\I_{1/n}$ is not Egorov (by Corollary \ref{corFsigma}). On the other hand, $\I\cap\J=\Fin\oplus\Fin$, which is isomorphic with $\Fin$, so Egorov. 
    \item It can happen that an intersection of an Egorov ideal with a non-Egorov ideal is not Egorov. To see it, consider the above ideal $\I$ and $\Fin\oplus\Fin$. 
    \item  It can happen that an intersection of two non-Egorov ideals is non-Egorov. This holds for instance for $\I=\J=\I_{1/n}$.
\end{itemize}
\end{remark}

\subsection{Products of the form \texorpdfstring{$\sum^\J_{n\in\omega}\I_n$}{}} 

\begin{lemma}
\label{Fubini-lem1}
Let $\J$ and $\I_n$, for $n\in\omega$, be ideals on $\omega$. If $\J$ is not Egorov, then $\sum^\J_{n\in\omega}\I_n$ is not Egorov.
\end{lemma}

\begin{proof}
It follows from Propositions \ref{RK-main} and \ref{RK-folklore}(b).
\end{proof}

\begin{lemma}
\label{Fubini-lem2}
Let $\J$ and $\I_n$, for $n\in\omega$, be ideals on $\omega$. If 
$$\{n\in\omega:\ \I_n\text{ is not Egorov}\}\notin\J,$$ 
then $\sum^\J_{n\in\omega}\I_n$ is not Egorov.
\end{lemma}

\begin{proof}
Denote $Z=\{n\in\omega:\ \I_n\text{ is not Egorov}\}\notin\J$ and using Theorem \ref{onehalf}, for each $i\in Z$ find a sequence $(X_{(i,j)})_{j\in\omega}$ of measurable subsets of $[0,1]$ such that $(\chi_{X_{(i,j)}})_{j\in\omega}$ is $\I_i$-pointwise convergent to $0$, but for every measurable $M\subseteq[0,1]$, if $\lambda(M)\leq\frac{1}{2}$, then $T_i(M)=\{j\in\omega:\ X_{(i,j)}\not\subseteq M\}\notin\I_i$. Put $X_{(i,j)}=\emptyset$, for all $i\in\omega\setminus Z$ and $j\in\omega$. Observe that $(\chi_{X_{(i,j)}})_{(i,j)\in\omega^2}$ is $\sum^\J_{n\in\omega}\I_n$-pointwise convergent to $0$.

Fix now a measurable $M\subseteq[0,1]$ such that $\lambda(M)\leq\frac{1}{2}$. To finish the proof, note that:
$$\left\{(i,j)\in\omega^2:\ X_{(i,j)}\not\subseteq M\right\}=\bigcup_{i\in Z}\{i\}\times T_i(M)\notin\sum^\J_{n\in\omega}\I_n.$$ 
\end{proof}

\begin{lemma}
\label{Fubini-lem3}
If $\J$ is an Egorov ideal on $\omega$ and $\I_n$, for $n\in\omega$, are analytic ideals such that $\{n\in\omega:\ \I_n\text{ is not Egorov}\}\in\J$, then $\sum^\J_{n\in\omega}\I_n$ is Egorov.
\end{lemma}

\begin{proof}
We will apply Theorem \ref{characteristic}. Fix $\eta>0$ and a sequence $(X_{i,j})_{(i,j)\in\omega^2}$ of measurable subsets of $[0,1]$ such that $(\chi_{X_{i,j}})_{(i,j)\in\omega^2}$ is $\sum^\J_{n\in\omega}\I_n$-pointwise convergent to $0$. 

Denote $Z=\{n\in\omega:\ \I_n\text{ is not Egorov}\}\in\J$. For each $i\in\omega\setminus Z$ put: 
$$Y_i=\{x\in[0,1]:\ (\chi_{X_{i,j}}(x))_{j\in\omega}\text{ is }\I_i\text{-pointwise convergent to }0\}=$$
$$=\{x\in[0,1]:\ \{j\in\omega:\ x\in X_{i,j}\}\in\I_i\}.$$

Observe that $Y_i$ is measurable. Indeed, the function $g_i:[0,1]\to\cP(\omega)$ given by $g_i(x)=\{j\in\omega:\ x\in X_{i,j}\}$ is measurable, so there is a Borel measure zero set $B_i\subseteq[0,1]$ such that $g_i\restriction([0,1]\setminus B_i)$ is Borel. Since $\I_i$ is analytic, the set $\{x\in [0,1]\setminus B_i:\ g_i(x)\in\I_i\}$ is analytic and we can conclude that $Y_i$ is measurable as a union of $\{x\in [0,1]\setminus B_i:\ g_i(x)\in\I_i\}$ and some subset of the measure zero set $B_i$. 

Put $Y_i=[0,1]$ for all $i\in Z$. Since $Z\in\J$ and $(\chi_{X_{i,j}})_{(i,j)\in\omega^2}$ is $\sum^\J_{n\in\omega}\I_n$-pointwise convergent to $0$, $(\chi_{Y_i^c})_{i\in\omega}$ is $\J$-pointwise convergent to $0$. Since $\J$ is Egorov, there is a measurable $M\subseteq[0,1]$ such that $\lambda(M)\leq\frac{\eta}{2}$ and $A=\{i\in\omega:\ Y_i^c\not\subseteq M\}\in\J$. 

For each $i\in \omega\setminus Z$ and $j\in\omega$ put $Y^i_j=X_{i,j}\cap Y_i$. Then $(\chi_{Y^i_j})_{j\in\omega}$ is $\I_i$-pointwise convergent to $0$. Since $i\in \omega\setminus Z$, $\I_i$ is Egorov and we can find a measurable $M_i\subseteq[0,1]$ such that $\lambda(M_i)\leq\frac{\eta}{2^{i+2}}$ and $A_i=\{j\in\omega:\ X_{i,j}\cap Y_i\not\subseteq M_i\}\in\I_i$. 

Define $N=M\cup\bigcup_{i\in\omega\setminus Z}M_i$. Then $\lambda(N)\leq \frac{\eta}{2}+\sum_{i\in\omega}\frac{\eta}{2^{i+2}}=\eta$. Moreover,
if $i\in\omega\setminus Z$, then $X_{i,j}\not\subseteq N$ implies that either $X_{i,j}\cap Y_i\not\subseteq N$ or $X_{i,j}\cap Y^c_i\not\subseteq N$. However, the latter is possible only if $i\in A$. Thus:
$$\{(i,j)\in\omega:\ X_{i,j}\not\subseteq N\}\subseteq ((A\cup Z)\times\omega)\cup\bigcup_{i\in\omega\setminus (A\cup Z)}\{i\}\times A_i\in \sum^\J_{n\in\omega}\I_n.$$
\end{proof}

\begin{remark}
\label{rem-Fubini}
In Lemma \ref{Fubini-lem2} we applied Theorem \ref{onehalf}, which uses regularity of Lebesgue'a measure. Thus, we do not know if all results from this Subsection are valid for Egorov ideals in the sense of Kadets and Leonov. However, the proof of Lemma \ref{Fubini-lem3} in the case of Borel ideals $\I_i$ can be simplified, as measurability of the set $Y_i=\{x\in[0,1]:\ \{j\in\omega:\ x\in X_{i,j}\}\in\I_i\}$ for Borel ideal $\I_i$ is obvious. Thus, for Borel ideals we do not need any special properties of measure and we see that in this case Lemma \ref{Fubini-lem3} is true also for Egorov ideals in the sense of Kadets and Leonov.
\end{remark}

\begin{theorem}
\label{Fubini-thm}
Let $\J$ and $\I_n$, for $n\in\omega$, be ideals on $\omega$. If all $\I_n$ are analytic, then:
$$\sum^\J_{n\in\omega}\I_n\text{ is Egorov }\Longleftrightarrow(\J\text{ is Egorov and }\{n\in\omega:\ \I_n\text{ is not Egorov}\}\in\J).$$
\end{theorem}

\begin{proof}
It follows from Lemmas \ref{Fubini-lem1}, \ref{Fubini-lem2} and \ref{Fubini-lem3}.
\end{proof}

\begin{corollary}
\label{Fubini-cor}
Let $\I$ and $\J$ be ideals on $\omega$. If $\J$ is analytic, then:
$$\I\otimes\J\text{ is Egorov }\Longleftrightarrow(\I\text{ is Egorov and }\J\text{ is Egorov}).$$
\end{corollary}

\begin{proof}
It follows from Theorem \ref{Fubini-thm}.
\end{proof}

In Remark \ref{rem-ex} we have seen examples of Egorov and non-Egorov ideals. Now we give two more examples of Egorov ideals, that have been considered in the literature in completely different contexts. 

\begin{example}
\label{BI+CEX}
The following ideals are Egorov:
\begin{itemize}
    \item[(a)] $\mathcal{BI}$ defined in \cite{Unboring};
    \item[(b)] $\mathcal{CEI}$ defined in \cite{DebsConjecture}.
\end{itemize}
\end{example}

\begin{proof}
It follows from Propositions \ref{Itimesempty}, \ref{prop:Egorov-sumIn}, \ref{prop:Egorov-cap} and Corollary \ref{Fubini-cor}, as $\mathcal{BI}=(\{\emptyset\}\otimes\Fin^2)\cap(\Fin\otimes[\omega^2]^{<\omega})$ and $\mathcal{CEI}=(\{\emptyset\}\otimes\Fin^3)\cap(\Fin^3\otimes\{\emptyset\})$.
\end{proof}

\section{Uncountably many Borel Egorov ideals}
\label{sec:uncountable}

In this section we will show that there are $2^\omega$ pairwise non-isomorphic Borel Egorov ideals. We start with two simple observations.

\begin{remark}
There are $2^\omega$ pairwise non-isomorphic non-Egorov Borel (even $\bf{\Sigma^0_2}$) ideals. To see it, recall that by \cite[Theorem 1]{MezaGuzman}, there are $2^\omega$ many non-isomorphic tall summable (hence, non-pathological $\bf{\Sigma^0_2}$) ideals and apply Corollary \ref{corFsigma}, to conclude that all of them have to be non-Egorov.
\end{remark}

\begin{remark}
There are $2^\omega$ pairwise non-equal Egorov Borel (even $\bf{\Sigma^0_2}$) ideals. To see it, let $\cA\subseteq\cP(\omega)$ be any almost disjoint family of cardinality $2^\omega$ consisting of infinite sets and for each $A\in\cA$ define $\I_A=\{M\subseteq\omega:\ M\setminus A\in\Fin\}$. It is easy to see that each $\I_A$ is isomorphic to $\Fin\oplus\cP(\omega)$, hence $\bf{\Sigma^0_2}$ and Egorov (by \cite[Corollary 2.9]{KadetsLeonov}). To see that $\I_A\neq\I_B$ whenever $A,B\in\cA$, $A\neq B$, it suffices to observe that $A\in\I_A$ and $A\notin\I_B$. 
\end{remark}

\begin{theorem}
\label{thm-uncountably}
There are $2^\omega$ pairwise non-isomorphic Egorov $\bf{\Pi^0_\omega}$ ideals.
\end{theorem}

\begin{proof}
Let $\cA\subseteq\cP(\omega)$ be any almost disjoint family of cardinality $2^\omega$ consisting of infinite sets and for each $A\in\cA$ define: 
$$\I_A=\{M\subseteq\sum_{n\in\omega}\omega^{n+1}:\ \{n\in\omega:\ M_{(n)}\neq\emptyset\}\setminus A\in\Fin\}$$
and $\J_A=(\sum_{n\in\omega}\Fin^{n+1})\cap \I_A$. Then each $\J_A$ is Egorov, by Propositions \ref{prop:Egorov-sumIn} and \ref{prop:Egorov-cap} together with Remark \ref{rem-ex}.

Note that  
$$\sum_{n\in\omega}\Fin^{n+1}=\{M\subseteq\sum_{n\in\omega}\omega^{n+1}:\ \forall_{n\in\omega}\ M\in\phi_n^{-1}[\Fin^{n+1}]\},$$
where $\phi_n:\cP(\sum_{n\in\omega}\omega^{n+1})\to\cP(\omega^{n+1})$ given by $\phi_n(M)=M_{(n)}$ is continuous. As each $\Fin^n$ is $\bf{\Sigma^0_{2n}}$ (see \cite[Proposition 6.4]{Debs}), we conclude that $\sum_{n\in\omega}\Fin^{n+1}$ is $\bf{\Pi^0_\omega}$. Consequently, also $\J_A$ is $\bf{\Pi^0_\omega}$, as each $\I_A$ is $\bf{\Sigma^0_2}$ (even countably generated, as it is isomorphic to $\Fin\oplus\mathcal{P}(\omega)$). 

Fix any $A,B\in\cA$ with $A\neq B$. We will show that $\J_A$ and $\J_B$ are not isomorphic. Let $f:\sum_{n\in\omega}\omega^{n+1}\to \sum_{n\in\omega}\omega^{n+1}$ be any bijection. Note that for each $n\in\omega$ there are three possibilities:
\begin{itemize}
    \item[(a)] $f(n,i_n)\in\{n\}\times\omega^{n+1}$, for some $i_n\in\omega^{n+1}$;
    \item[(b)] $(f^{-1}[\{k_n\}\times\omega^{k_n+1}])_{(n)}\in (\Fin^{n+1})^+$, for some $k_n\neq n$;
    \item[(c)] $(f^{-1}[\{k\}\times\omega^{k+1}])_{(n)}\neq\emptyset$ for infinitely many $k\neq n$. 
\end{itemize}

Suppose first that item (b) holds for at least one $n\in\omega$ and denote $W=(f^{-1}[\{k_n\}\times\omega^{k_n+1}])_{(n)}\in (\Fin^{n+1})^+$. If $(f[\{n\}\times W])_{(k_n)}\in\Fin^{k_n+1}$, then $\{k_n\}\times (f[\{n\}\times W])_{(k_n)}\in\J_B$, but $f^{-1}[\{k_n\}\times (f[\{n\}\times W])_{(k_n)}]=\{n\}\times W\notin\J_A$. If $(f[\{n\}\times W])_{(k_n)}\in(\Fin^{k_n+1})^+$, then by \cite[Remark after Proposition 2.9]{homogeneous}, $\Fin^{n+1}\restriction W$ is isomorphic with $\Fin^{n+1}$ and $\Fin^{k_n+1}\restriction (f[\{n\}\times W])_{(k_n)}$ is isomorphic with $\Fin^{k_n+1}$. Since $k_n\neq n$, $\Fin^{n+1}$ and $\Fin^{k_n+1}$ are not isomorphic (see \cite[Section 6]{Debs}), so either there is $M\in\Fin^{n+1}$ such that $(f[\{n\}\times M])_{(k_n)}\notin\Fin^{k_n+1}$ or there is $M\notin\Fin^{n+1}$ such that $(f[\{n\}\times M])_{(k_n)}\in\Fin^{k_n+1}$. In the former case, $\{n\}\times M\in\J_A$, but $f[\{n\}\times M]\notin\J_B$ and in the latter case, $\{n\}\times M\notin\J_A$, but $f[\{n\}\times M]\in\J_B$.

From now on we will assume that for each $n$ item (b) does not hold. Denote $T=\{n\in\omega:\ \text{item (a) holds for }n\}$ and $S=\{n\in\omega:\ \text{item (c) holds for }n\}$. Observe that $T\cup S=\omega$, so there are two possible cases: either $T\cap A$ is infinite or $S\cap A$ is infinite.

Assume first that $T\cap A$ is infinite and define $M=\{(n,i_n):\ n\in T\cap A\}$. Observe that $M\in\J_A$, but $f[M]\notin\J_A$, as $f[M]\notin\I_A$.

Assume now that $S\cap A$ is infinite and for each $n\in S\cap A$ let $X_n=\{k\in\omega\setminus\{n\}:\ (f^{-1}[\{k\}\times\omega^{k+1}])_{(n)}\neq\emptyset\}$. There are two subcases.

If $X_n\setminus B\in\Fin$ for some $n\in S\cap A$, then $\{n\}\times\omega^{n+1}\notin\J_A$, but $f[\{n\}\times\omega^{n+1}]=\sum_{k\in X_n}\{k\}\times(f[\{n\}\times\omega^{n+1}])_{(k)}\in\J_B$ (as (b) does not hold for $n$).

If $X_n\setminus B$ is infinite for each $n\in S\cap A$, then inductively pick points $(a_n,b_n)\in \sum_{n\in\omega}\omega^{n+1}$, for all $n\in S\cap A$, such that:
\begin{itemize}
    \item $a_n\in X_n\setminus B$, for all $n\in S\cap A$;
    \item $a_n\neq a_m$, for all $n,m\in S\cap A$, $n\neq m$;
    \item $(a_n,b_n)\in f[\{n\}\times\omega^{n+1}]$, for all $n\in S\cap A$.
\end{itemize}
Define $M=\{(a_n,b_n):\ n\in S\cap A\}$. Then $M\notin\J_B$ (as $M\notin\I_B$), but $f^{-1}[M]\in\J_A$. This finishes the entire proof.
\end{proof}

\begin{corollary}
\label{GenEgorov-cor2}
There are $2^\omega$ pairwise non-isomorphic $\bf{\Pi^0_\omega}$ ideals that are:
\begin{itemize}
    \item Egorov in the sense of Kadets and Leonov, i.e., for every finite measure space $(X,\cM,\mu)$, if $\eta>0$ and $(f_n)_{n\in\omega}$ is a sequence of measurable real-valued functions defined on $X$ which is $\I$-pointwise convergent to a measurable function $f:X\to\mathbb{R}$, then there is $M\in\cM$ such that $\mu(X\setminus M)\leq\eta$ and $(f_n\restriction M)_{n\in\omega}$ is $\I$-uniformly convergent to $f\restriction M$;
    \item generalized Egorov in the sense of Korch and Repick\'{y} under $\non(\cN)<\bb$, i.e., if $\eta>0$ and $(f_n)_{n\in\omega}$ is a sequence of (non-necessarily measurable) real-valued functions defined on $[0,1]$ which is $\I$-pointwise convergent to a (non-necessarily measurable) function $f:[0,1]\to\mathbb{R}$, then there is $M\subseteq[0,1]$ such that $\lambda^\star(M)\geq 1-\eta$ (here $\lambda^\star$ is the Lebesgue's outer measure on $[0,1]$) and $(f_n\restriction M)_{n\in\omega}$ is $\I$-uniformly convergent to $f\restriction M$.
\end{itemize}
\end{corollary}

\begin{proof}
Consider the ideals $\J_A$ from the proof of Theorem \ref{thm-uncountably}. We already know that they are pairwise non-isomorphic and $\bf{\Pi^0_\omega}$. 

To see that under $\non(\cN)<\bb$ each $\J_A$ is generalized Egorov in the sense of Korch and Repick\'{y}, apply \cite[Theorem 1.3(2) and Theorem 3.3(1,4,5,7)]{Repicky}. 

To see that each $\J_A$ is Egorov in the sense of Kadets and Leonov, start with noting that all $\I_A$ are Egorov in the sense of Kadets and Leonov (by \cite[Corollary 2.9]{KadetsLeonov}). Moreover, if two ideals, $\I_0$ and $\I_1$, are Egorov in the sense of Kadets and Leonov, then so is their intersection. Indeed, given any finite measure space $(X,\cM,\mu)$, $\eta>0$ and a sequence $(f_n)_{n\in\omega}$ of measurable real-valued functions defined on $X$ which is $(\I_0\cap\I_1)$-pointwise convergent to a measurable function $f:X\to\mathbb{R}$, we can find $M_0,M_1\in\cM$ such that $\mu(X\setminus M_i)\leq\frac{\eta}{2}$ and $(f_n\restriction M_i)_{n\in\omega}$ is $\I_i$-uniformly convergent to $f\restriction M_i$, for $i=0,1$. Then $M_0\cap M_1\in\cM$, $\mu(X\setminus (M_0\cap M_1))\leq\eta$ and $(f_n\restriction (M_0\cap M_1))_{n\in\omega}$ is $(\I_0\cap\I_1)$-uniformly convergent to $f\restriction (M_0\cap M_1)$.

Finally, $\sum_{n\in\omega}\Fin^{n+1}$ is Egorov in the sense of Kadets and Leonov. Indeed, for products of the form $\sum_{n\in\omega}\I_n$ the same reasoning as in the proof of Proposition \ref{prop:Egorov-sumIn} works and each $\Fin^{n+1}$ is Egorov in the sense of Kadets and Leonov by Remark \ref{rem-Fubini} and \cite[Corollary 2.9]{KadetsLeonov}.
\end{proof}


\bibliographystyle{amsplain}
\bibliography{Egorov-references.bib}

\end{document}